\documentclass[11pt,twoside,english]{article}
\usepackage{mathptmx}
\usepackage[T1]{fontenc}
\usepackage[utf8]{inputenc}
\usepackage[a4paper]{geometry}
\usepackage{color}
\usepackage{babel}
\usepackage{amsmath}
\usepackage{amsthm}
\usepackage{amssymb}
\usepackage{esint}
\usepackage[unicode=true,
 bookmarks=true,bookmarksnumbered=true,bookmarksopen=false,
 breaklinks=false,pdfborder={0 0 1},backref=false,colorlinks=true]
 {hyperref}
\hypersetup{
 pdfauthor={G Xi}}

\makeatletter
\numberwithin{equation}{section}
\theoremstyle{plain}
\newtheorem{thm}{\protect\theoremname}[section]
\theoremstyle{definition}
\newtheorem{defn}[thm]{\protect\definitionname}
\theoremstyle{plain}
\newtheorem{lem}[thm]{\protect\lemmaname}
\theoremstyle{remark}
\newtheorem*{rem*}{\protect\remarkname}
\theoremstyle{remark}
\newtheorem{rem}[thm]{\protect\remarkname}
\theoremstyle{plain}
\newtheorem{prop}[thm]{\protect\propositionname}

\usepackage{babel}
\allowdisplaybreaks[4]

\date{}

  \providecommand{\definitionname}{Definition}
  \providecommand{\lemmaname}{Lemma}
\providecommand{\theoremname}{Theorem}

\makeatother

\providecommand{\definitionname}{Definition}
\providecommand{\lemmaname}{Lemma}
\providecommand{\propositionname}{Proposition}
\providecommand{\remarkname}{Remark}
\providecommand{\theoremname}{Theorem}

\begin{document}
\global\long\def\divg{{\rm div}\,}%

\global\long\def\curl{{\rm curl}\,}%

\global\long\def\rt{\mathbb{R}^{3}}%

\global\long\def\rd{\mathbb{R}^{d}}%

\title{On a maximal inequality and its application to SDEs with singular
drift}
\author{Xuan Liu\thanks{Normura International, 30/FL Two International Finance Centre, Hong
Kong. Email: chamonixliu@163.com} $\ $and Guangyu Xi\thanks{Department of Mathematics, University of Maryland, College Park, MD
20742, USA. Email: gxi@umd.edu}}
\maketitle
\begin{abstract}
In this paper we present a Doob type maximal inequality for stochastic
processes satisfying the conditional increment control condition.
If we assume, in addition, that the margins of the process have uniform
exponential tail decay, we prove that the supremum of the process
decays exponentially in the same manner. Then we apply this result
to the construction of the almost everywhere stochastic flow to stochastic
differential equations with singular time dependent divergence-free
drift.

\medskip


\textit{MSC2010 subject classifications: 60E15, 60H10}
\end{abstract}

\section{Introduction}

Let $\{\Omega,\{\mathcal{F}_{t}\}_{t\geq0},\mathcal{F},\mathbb{P}\}$
be a filtered probability space satisfying the usual conditions. Let
$\{X_{t}\}_{t\in[0,T]}$ be an $\{\mathcal{F}_{t}\}$-adapted stochastic
process. There has been abundant research on the distribution of the
supremum $\sup_{t\in[0,T]}\vert X_{t}\vert$ since Doob's martingale
maximal inequality. See, for example, \cite{Bogachev1998,McLeish1975,Talagrand1996,Talagrand2014}.
Here we consider continuous processes $\{X_{t}\}_{t\in[0,T]}$ satisfying
the conditional increment control condition as follows.
\begin{defn}
Let $\{X_{t}\}_{t\in[0,T]}$ be a continuous $\{\mathcal{F}_{t}\}$-adapted
stochastic process, and let $p>1$, $0<h\leq1$, $ph>1$. $\{X_{t}\}_{t\in[0,T]}$
is said to satisfy the conditional increment control with parameter
$(p,h)$ if $X_{t}\in L^{p}(\Omega,\rd)$ for all $t\in[0,T]$ and
there exists a constant $A_{p,h}\geq0$ independent of $s$ and $t$
such that 
\begin{equation}
\mathbb{E}\left[\left|\mathbb{E}(X_{t}\vert\mathcal{F}_{s})-X_{s}\right|^{p}\right]\leq A_{p,h}\vert t-s\vert^{ph},\qquad\mbox{ for all }0\leq s<t\leq T.\label{eq: conditional increment control}
\end{equation}
\end{defn}

For processes satisfying condition (\ref{eq: conditional increment control}),
we prove a Doob type maximal inequality as follows.
\begin{thm}
\label{Theorem: Doob Type Maximal Inequality}Suppose $\{X_{t}\}_{t\in[0,T]}$
is a continuous $\{\mathcal{F}_{t}\}$-adapted process satisfying
condition (\ref{eq: conditional increment control}). Let $0\leq s_{0}<t_{0}\leq T$,
and $X^{\ast}=\sup_{u\in[s_{0},t_{0}]}\vert X_{u}\vert$. Then for
any $1<q\leq p$, 
\[
\Vert X^{\ast}\Vert_{L^{q}}\leq\frac{q}{q-1}\left[C_{p,h}^{1/p}A_{p,h}^{1/p}\vert t_{0}-s_{0}\vert^{h}+\Vert X_{t_{0}}\Vert_{L^{q}}\right]
\]
for some constant $C_{p,h}>0$.
\end{thm}

Under condition (\ref{eq: conditional increment control}), we further
study the tail decay of $\sup_{t\in[0,T]}\vert X_{t}\vert$ when the
margins of $\{X_{t}\}_{t\in[0,T]}$ have uniform $\alpha$-exponential
decay for some $\alpha>0$, i.e. there exist $C_{1},C_{2}>0$ such
that
\[
\mathbb{P}\left(\vert X_{t}\vert\geq\lambda\right)\leq C_{2}\exp\left(-C_{1}\lambda^{\alpha}\right),\qquad\mbox{ for all }\lambda>0\mbox{ and all }t\in[0,T].
\]
In Theorem \ref{thm: Main theorem 1}, we prove that if a continuous
process $\{X_{t}\}_{t\in[0,T]}$ has uniform $\alpha$-exponential
marginal decay and satisfies the conditional increment control for
$(p,h)$ with $p>1$, $0<h\leq1$ and $ph>1$, then its supremum $\sup_{t\in[0,T]}\vert X_{t}\vert$
decays in the same manner as its margins, i.e. 
\[
\mathbb{P}\left(\sup_{t\in[0,T]}\vert X_{t}\vert\geq\lambda\right)\leq C\exp(-C\lambda^{\alpha}).
\]

Our results here are closely related to the celebrated theorems of
Kolmogorov and Doob. The conditional increment control condition can
be easily deduced from Kolmogorov’s continuity and tightness criteria,
which means our result can be applied to a large class of diffusion
processes as well as Gaussian processes like fractional Brownian motion.
Moreover, the conditional increment control condition is also satisfied
by martingales, which makes Theorem \ref{Theorem: Doob Type Maximal Inequality}
a generalization of Doob's maximal inequality. In addition to being
mathematically interesting, our results also have practical significance
since we have no structural assumption on the processes. We only assume
the conditional increment control and exponential marginal decay,
which can be directly verified using empirical data. In this article,
we show an application of our results to the study of stochastic differential
equations 
\begin{equation}
dX_{t}=b(t,X_{t})dt+dB_{t},\label{eq: SDE1}
\end{equation}
where $b$ is a time dependent divergence-free vector field and $B_{t}$
is the standard Brownian motion on $\rd$. Our main result of this
application is Theorem \ref{thm: SDE main theorem}, which states
that there is unique almost everywhere stochastic flow $X(\omega,x):\Omega\times\rd\rightarrow C([0,T],\rd)$
to SDE (\ref{eq: SDE1}) if $b\in L^{1}(0,T;W^{1,p}(\rd))\cap L^{l}(0,T;L^{q}(\rd))$
with $d\geq3$, $p\geq1$, $\frac{2}{l}+\frac{d}{q}\in[1,2)$. Here
$X(\omega,x)$ is defined for almost every $(\omega,x)\in\Omega\times\rd$
under $\mathbb{P}\times m$, where $m$ is the Lebesgue measure on
$\rd$.

For SDEs with singular drift, Aronson \cite{Aronson1968} proved that
there is a unique weak solution to (\ref{eq: SDE1}) when $b\in L^{l}(0,T;L^{q}(\rd))$
with $\frac{2}{l}+\frac{d}{q}<1$. Moreover, it was proved that the
transition probability of $\{X_{t}\}_{t\in[0,T]}$ satisfies the Aronson
estimate 
\[
\frac{1}{C_{1}(t-\tau)^{d/2}}\exp\left(-C_{2}\left(\frac{\vert x-\xi\vert^{2}}{t-\tau}\right)\right)\leq\Gamma(t,x;\tau,\xi)\leq\frac{C}{(t-\tau)^{d/2}}\exp\left(-\frac{1}{C}\left(\frac{\vert x-\xi\vert^{2}}{t-\tau}\right)\right).
\]
Actually the Aroson estimate is true in more general cases when the
diffusion coefficient is uniformly elliptic and $b\in L^{l}(0,T;L^{q}(\rd))$
with $\frac{2}{l}+\frac{d}{q}\leq1$, $q>d$. With the same condition
on $b$ as in Aronson \cite{Aronson1968}, Krylov and Röckner \cite{KrylovRockner2005}
proved that there exists a unique strong solution to (\ref{eq: SDE1}).
Regularity results about the strong solution are obtained in Fedrizzi
and Flandoli \cite{FedrizziFlandoli2011,FedrizziFlandoli2013}. If
we assume boundedness of $\divg b$, since the introduction of the
renormalized solutions by DiPerna and Lions \cite{DiPernaLions1989},
there has been lots of work on ODEs $dX_{t}=b(t,X_{t})dt$. In particular,
Crippa and De Lellis \cite{DelellisCrippa2008} developed new estimate
on ODEs with Sobolev coefficient $b$ and gave a new approach to construct
the Diperna-Lions flow. This idea is extended to solve SDEs in \cite{FangLuoThalmaier2010,ZhangXicheng2009,ZhangXicheng2010}.
In Zhang \cite{ZhangXicheng2009}, in addition to boundedness of $\divg b$,
it assumes that $\nabla b\in L\log L(\rd)$ to control $X_{t}$ locally
and that $\vert b\vert/(1+\vert x\vert)\in L^{\infty}(\rd)$ to control
$X_{t}$ from explosion. Together with Sobolev condition on the diffusion
coefficient, Zhang \cite{ZhangXicheng2009} proved existence of a
unique almost everywhere stochastic flow to SDEs. Since it is harder
to control the growth of solutions to SDEs than ODEs, the linear growth
condition on $b$ is needed in Zhang \cite{ZhangXicheng2009}, while
it is not necessary in Crippa and De Lellis \cite{DelellisCrippa2008}.
Fang, Luo and Thalmaier \cite{FangLuoThalmaier2010} extend it to
SDEs in Gaussian space with Sobolev diffusion and drift coefficients.
In Zhang \cite{ZhangXicheng2010}, it relaxes the boundedness of $\divg b$
to only the negative part of $\divg b$, and proved large deviation
principle for the corresponding SDEs. 

In Section 3, we prove the existence of the almost everywhere stochastic
flow to (\ref{eq: SDE1}) through approximation. Take smooth approximation
sequence $b_{n}\rightarrow b$, such that $\{b_{n}\}_{n\in\mathbb{N}}$
is uniformly bounded in $L^{1}(0,T;W^{1,p}(\rd))\cap L^{l}(0,T;L^{q}(\rd))$
and divergence-free. Using the Aronson type upper bound estimate of
the transition probability proved in Qian and Xi \cite[Corollary 9]{qian2017parabolic},
we can show that $\{X_{t}^{(n)}\}_{n\in\mathbb{N}}$ satisfies uniform
conditional increment control and exponential marginal decay. Hence
Theorem \ref{thm: Main theorem 1} implies that $\{\sup_{t\in[0,T]}\vert X_{t}^{(n)}\vert\}_{n\in\mathbb{N}}$
can be controlled uniformly and allows us to remove the linear growth
condition on $b$ used in Zhang \cite{ZhangXicheng2009}. Then following
the idea of Zhang \cite{ZhangXicheng2009}, in Theorem \ref{thm: SDE main theorem}
we prove that the sequence $\{X_{t}^{(n)}\}_{n\in\mathbb{N}}$ converges
to a unique limit $X_{t}$, which is the unique almost everywhere
stochastic flow to SDE (\ref{eq: SDE1}). It worth noting that the
proof of Theorem \ref{thm: SDE main theorem} only uses the moment
estimate of the supremum $\sup_{t\in[0,T]}\vert X_{t}\vert$ proved
in Proposition \ref{prop: supremum estimate of diffusion}. The moment
estimate actually can be obtained from the Aronson estimate using
Kolmogorov's continuity theorem as well, while the exponential decay
of the supremum proved in Proposition \ref{prop: supremum estimate of diffusion}
is new. In a special case, Theorem \ref{thm: SDE main theorem} is
true when divergence-free $b\in L^{2}(0,T;H^{1}(\rt))$, which is
of particular interest since the Leray-Hopf weak solutions to the
3-dimensional Navier-Stokes equations are in this space. However,
the existence of a unique almost everywhere stochastic flow does not
imply the uniqueness of the weak solutions to the corresponding parabolic
equations. Since the stochastic flow is defined for almost everywhere
initial data $x\in\rd$, it actually disguises the ``bad'' points
in the measure zero set. For more discussion on the non-uniqueness
of parabolic equations, we refer to Modena and Székelyhidi \cite{StefanoSzekelyhidi2018}.
For simplicity, in this article we only discuss the case when the
drift is divergence-free and the diffusion is the standard Brownian
motion. But actually, using the idea in \cite[Corollary 9]{qian2017parabolic},
we can obtain the Aronson type estimate and hence extend the result
in Theorem \ref{thm: SDE main theorem} for $\divg b\in L^{l'}(0,T;L^{q'}(\rd))$
with $\frac{2}{l'}+\frac{d}{q'}<2$. Moreover, the diffusion coefficient
also can be extended to be Sobolev as in Zhang \cite{ZhangXicheng2009}.

\section{Maximal Inequality}

In this section, we always assume that the process $\{X_{t}\}_{t\in[0,T]}$
satisfies the conditional increment control (\ref{eq: conditional increment control})
with $p>1$, $0<h\leq1$ and $ph>1$. Under this condition, we prove
the Doob type maximal inequality in Theorem \ref{Theorem: Doob Type Maximal Inequality}
and the exponential tail decay for the supremum $\sup_{t\in[0,T]}\vert X_{t}\vert$
in Theorem \ref{thm: Main theorem 1}. Before starting the proof,
we give below some examples of processes which satisfy the conditional
increment control.
\begin{itemize}
\item Any continuous martingales. The conditional increment control is satisfied
for any $(p,h)$ with $A_{p,h}=0$. 
\item Fractional Brownian motions with Hurst parameter $h\in(0,1)$.
\item Let $\{X_{t}\}_{t\in[0,T]}$ be a continuous stochastic process satisfying
\begin{equation}
\mathbb{E}\left(\left|X_{t}-X_{s}\right|^{p}\right)\leq A_{p,h}\vert t-s\vert^{ph},\qquad\mbox{ for all }0\leq s<t\leq T.\label{eq: conditional increment 2}
\end{equation}
Using Jensen's inequality, it is easy to see that the conditional
increment control is satisfied with the same parameter $(p,h)$ and
the same constant $A_{p,h}$. Processes satisfying (\ref{eq: conditional increment 2})
are archetypal examples considered in the rough paths theory for which
canonical constructions of associated geometric rough paths are available
and well-studied (see \cite{MR2604669,MR2036784}). This type of processes
also arises as solutions to SDEs.
\end{itemize}

\subsection{A Doob Type Maximal Inequality}

Before proving the maximal inequality, we will need two lemmas. Firstly,
we prove the following estimate for the supremum of the conditional
increment. 
\begin{lem}
\label{lem: supremum p norm}Suppose $\{X_{t}\}_{t\in[0,T]}$ is a
continuous $\{\mathcal{F}_{t}\}$-adapted process satisfying conditional
increment control with parameter $(p,h)$. For any $0\leq s_{0}<t_{0}\leq T$,
there holds that 
\[
\mathbb{E}\left(\sup_{s_{0}\leq s<t\leq t_{0}}\left|\mathbb{E}(X_{t}\vert\mathcal{F}_{s})-X_{s}\right|^{p}\right)\leq C_{p,h}A_{p,h}\vert t_{0}-s_{0}\vert^{ph},
\]
where 
\[
C_{p,h}=[2\zeta(\theta)]^{p-1}\left(\frac{p}{p-1}\right)^{p}\left(\frac{4}{ph-1}\right)^{\theta(p-1)+1}\Gamma[\theta(p-1)+1].
\]
Here $\theta>1$ is an arbitrary constant, $\zeta(\theta):=\sum_{m=1}^{\infty}m^{-\theta}<\infty$
for all $\theta>1$  and $\Gamma(z)$ is the Gamma function.
\end{lem}

\begin{proof}
Let $s,t\in[s_{0},t_{0}]$, $s<t$ be fixed temporarily. Denote the
dyadic intervals
\[
I_{l}^{m}=[t_{l-1}^{m},t_{l}^{m}]=s_{0}+(t_{0}-s_{0})*\left[\frac{l-1}{2^{m}},\frac{l}{2^{m}}\right].
\]
Then we will construct a sequence of intervals $\{J_{k}\}\subset\{I_{l}^{m}:1\leq l\leq2^{m},m\geq0\}$
which gives a partition to $[s,t]$ such that
\begin{description}
\item [{(i)}] $J_{k}$, $k=1,2,\cdots$, have mutually disjoint interior;
\item [{(ii)}] For any $m\geq1$, there are at most two elements of $\{J_{k}\}$
with length $(t_{0}-s_{0})2^{-m}$;
\item [{(iii)}] $(s,t)\subset\cup_{k=1}^{\infty}J_{k}\subset[s,t]$.
\end{description}
Suppose $m_{0}=\min\{m\in\mathbb{N}:\exists1\leq l\leq2^{m}\text{\mbox{ such that }}I_{l}^{m}\subset[s,t]\}$.
For $m_{0}$, either there is only one $1\leq l_{0}\leq2^{m_{0}}$
such that $I_{l_{0}}^{m_{0}}\subset[s,t]$ and $[s,t]=[s,t_{l_{0}-1}^{m_{0}}]\cup I_{l_{0}}^{m_{0}}\cup[t_{l_{0}}^{m_{0}},t]$,
or there are two consecutive $1\leq l_{0}<l_{0}+1\leq2^{m_{0}}$ such
that $\left(I_{l_{0}}^{m_{0}}\cup I_{l_{0}+1}^{m_{0}}\right)\subset[s,t]$
and $[s,t]=[s,t_{l_{0}-1}^{m_{0}}]\cup I_{l_{0}}^{m_{0}}\cup I_{l_{0}+1}^{m_{0}}\cup[t_{l_{0}+1}^{m_{0}},t]$.
Here we only deal with the first case and the second case follows
the same argument. Notice that $\vert s-t_{l_{0}-1}^{m_{0}}\vert$
and $\vert t_{l_{0}}^{m_{0}}-t\vert$ are smaller than $(t_{0}-s_{0})2^{-m_{0}}$.
For $[t_{l_{0}}^{m_{0}},t]$, we set $m_{k+1}=\min\{m>m_{k}:\exists1\leq l\leq2^{m}\text{\mbox{ such that }}I_{l}^{m}\subset[t_{l_{k}}^{m_{k}},t]\}$.
There is at most one $1\leq l_{k+1}\leq2^{m_{k+1}}$ such that $I_{l_{k+1}}^{m_{k+1}}\subset[t_{l_{k}}^{m_{k}},t]$
since $\vert t_{l_{k}}^{m_{k}}-t\vert<(t_{0}-s_{0})2^{-m_{k}}$. Then
$\{I_{l_{k}}^{m_{k}}\}$ forms a dyadic partition to $[t_{l_{0}}^{m_{0}},t]$,
together with $I_{l_{0}}^{m_{0}}$ and a dyadic partition to $[s,t_{l_{0}-1}^{m_{0}}]$
following similar argument, we obtained the collection of intervals
$\{J_{k}\}$ satisfying (i)-(iii).

Suppose $J_{k}=[u_{k-1},u_{k}]$. Then 
\begin{align*}
\left|\mathbb{E}(X_{t}\vert\mathcal{F}_{s})-X_{s}\right| & =\left|\sum_{k=1}^{\infty}\mathbb{E}(\Delta X_{J_{k}}\vert\mathcal{F}_{s})\right|\\
 & \leq\sum_{k=1}^{\infty}\mathbb{E}\left[\left.\vert\mathbb{E}(\Delta X_{J_{k}}\vert\mathcal{F}_{u_{k-1}})\vert\right|\mathcal{F}_{s}\right]\\
 & =\sum_{m=1}^{\infty}\sum_{\{J_{k}:\vert J_{k}\vert=(t_{0}-s_{0})2^{-m}\}}\mathbb{E}\left[\left.\vert\mathbb{E}(\Delta X_{J_{k}}\vert\mathcal{F}_{u_{k-1}})\vert\right|\mathcal{F}_{s}\right],
\end{align*}
where $\Delta X_{J_{k}}=X_{u_{k}}-X_{u_{k-1}}$. Let $\xi_{l}^{m}=\mathbb{E}(\Delta X_{I_{l}^{m}}\vert\mathcal{F}_{t_{l-1}^{m}})$
for $1\leq l\leq2^{m}$, $m=1,2,\cdots$. For any $\theta>1$, take
$\zeta(\theta)=\sum_{m=1}^{\infty}m^{-\theta}$ to be the Riemann
zeta function. By Jensen's inequality, we have
\begin{align*}
\left|\mathbb{E}(X_{t}\vert\mathcal{F}_{s})-X_{s}\right|^{p} & \leq\left(\sum_{m=1}^{\infty}\frac{1}{\zeta(\theta)m^{\theta}}\zeta(\theta)m^{\theta}\sum_{\{J_{k}:\vert J_{k}\vert=(t_{0}-s_{0})2^{-m}\}}\mathbb{E}\left[\left.\vert\mathbb{E}(\Delta X_{J_{k}}\vert\mathcal{F}_{u_{k-1}})\vert\right|\mathcal{F}_{s}\right]\right)^{p}\\
 & \leq\sum_{m=1}^{\infty}\frac{1}{\zeta(\theta)m^{\theta}}\left(\zeta(\theta)m^{\theta}\sum_{\{J_{k}:\vert J_{k}\vert=(t_{0}-s_{0})2^{-m}\}}\mathbb{E}\left[\left.\vert\mathbb{E}(\Delta X_{J_{k}}\vert\mathcal{F}_{u_{k-1}})\vert\right|\mathcal{F}_{s}\right]\right)^{p}\\
 & =\zeta(\theta)^{p-1}\sum_{m=1}^{\infty}m^{\theta(p-1)}\left(\sum_{\{J_{k}:\vert J_{k}\vert=(t_{0}-s_{0})2^{-m}\}}\mathbb{E}\left[\left.\vert\mathbb{E}(\Delta X_{J_{k}}\vert\mathcal{F}_{u_{k-1}})\vert\right|\mathcal{F}_{s}\right]\right)^{p}\\
 & \leq[2\zeta(\theta)]^{p-1}\sum_{m=1}^{\infty}m^{\theta(p-1)}\sum_{\{J_{k}:\vert J_{k}\vert=(t_{0}-s_{0})2^{-m}\}}\left(\mathbb{E}\left[\left.\vert\mathbb{E}(\Delta X_{J_{k}}\vert\mathcal{F}_{u_{k-1}})\vert\right|\mathcal{F}_{s}\right]\right)^{p}\\
 & \le[2\zeta(\theta)]^{p-1}\sum_{m=1}^{\infty}m^{\theta(p-1)}\sum_{l=1}^{2^{m}}\sup_{r\in[s_{0},t_{0}]}\left[\mathbb{E}\left(\left.\vert\xi_{l}^{m}\vert\right|\mathcal{F}_{r}\right)\right]^{p},
\end{align*}
where the inequality in the fourth line is due to property (ii) of
$\{J_{k}\}$. Notice that for $s\leq t_{2}\leq t_{1}$, by Jensen's
inequality we have
\begin{align*}
\mathbb{E}\left|\mathbb{E}(X_{t_{1}}\vert\mathcal{F}_{s})-\mathbb{E}(X_{t_{2}}\vert\mathcal{F}_{s})\right|^{p} & =\mathbb{E}\left|\mathbb{E}\left(\left.[\mathbb{E}(X_{t_{1}}\vert\mathcal{F}_{t_{2}})-X_{t_{2}}]\right|\mathcal{F}_{s}\right)\right|^{p}\\
 & \leq\mathbb{E}\left(\mathbb{E}\left(\left.\vert\mathbb{E}(X_{t_{1}}\vert\mathcal{F}_{t_{2}})-X_{t_{2}}\vert^{p}\right|\mathcal{F}_{s}\right)\right)\\
 & =\mathbb{E}\left[\vert\mathbb{E}(X_{t_{1}}\vert\mathcal{F}_{t_{2}})-X_{t_{2}}\vert^{p}\right]\\
 & \leq A_{p,h}\vert t_{1}-t_{2}\vert^{ph}.
\end{align*}
By Kolmogorov's continuity theorem, for any fixed $s\in[s_{0},t_{0}]$,
$\mathbb{E}(X_{t}\vert\mathcal{F}_{s})$ is a continuous process with
respect to $t$. Recall that the filtration satisfies the usual conditions.
$\mathbb{E}(X_{t}\vert\mathcal{F}_{s})$ can be regarded as a process
of $s$ for fixed $t$ and it has a càdlàg modification by Doob's
regularization theorem. Hence, $\sup_{s_{0}\leq s<t\leq t_{0}}\left|\mathbb{E}(X_{t}\vert\mathcal{F}_{s})-X_{s}\right|^{p}$
is measurable with respect to $\mathcal{F}$ and we have 
\[
\sup_{s_{0}\leq s<t\leq t_{0}}\left|\mathbb{E}(X_{t}\vert\mathcal{F}_{s})-X_{s}\right|^{p}\leq[2\zeta(\theta)]^{p-1}\sum_{m=1}^{\infty}m^{\theta(p-1)}\sum_{l=1}^{2^{m}}\sup_{r\in[s_{0},t_{0}]}\left[\mathbb{E}\left(\left.\vert\xi_{l}^{m}\vert\right|\mathcal{F}_{r}\right)\right]^{p}.
\]
By Doob's maximal inequality for martingales, 
\begin{align*}
\mathbb{E}\left(\sup_{s_{0}\leq s<t\leq t_{0}}\left|\mathbb{E}(X_{t}\vert\mathcal{F}_{s})-X_{s}\right|^{p}\right) & \leq[2\zeta(\theta)]^{p-1}\sum_{m=1}^{\infty}m^{\theta(p-1)}\sum_{l=1}^{2^{m}}\mathbb{E}\left(\sup_{r\in[s_{0},t_{0}]}\left[\mathbb{E}\left(\left.\vert\xi_{l}^{m}\vert\right|\mathcal{F}_{r}\right)\right]^{p}\right)\\
 & \leq[2\zeta(\theta)]^{p-1}\sum_{m=1}^{\infty}m^{\theta(p-1)}\sum_{l=1}^{2^{m}}\left(\frac{p}{p-1}\right)^{p}\mathbb{E}\left(\left[\mathbb{E}\left(\left.\vert\xi_{l}^{m}\vert\right|\mathcal{F}_{t_{0}}\right)\right]^{p}\right)\\
 & \leq[2\zeta(\theta)]^{p-1}\left(\frac{p}{p-1}\right)^{p}\sum_{m=1}^{\infty}m^{\theta(p-1)}\sum_{l=1}^{2^{m}}\mathbb{E}\left(\vert\xi_{l}^{m}\vert^{p}\right)\\
 & \leq A_{p,h}[2\zeta(\theta)]^{p-1}\left(\frac{p}{p-1}\right)^{p}\sum_{m=1}^{\infty}m^{\theta(p-1)}2^{m}\left(\frac{\vert t_{0}-s_{0}\vert}{2^{m}}\right)^{ph}\\
 & =C_{p,h}A_{p,h}\vert t_{0}-s_{0}\vert^{ph},
\end{align*}
where 
\[
C_{p,h}=[2\zeta(\theta)]^{p-1}\left(\frac{p}{p-1}\right)^{p}\sum_{m=1}^{\infty}m^{\theta(p-1)}2^{-m(ph-1)}.
\]
Notice that 
\begin{align*}
\sum_{m=1}^{\infty}m^{\theta(p-1)}2^{-m(ph-1)} & \leq\sum_{m=1}^{\infty}\left(e^{\theta(p-1)}\int_{m-1}^{m}r^{\theta(p-1)}dr\right)2^{-m(ph-1)}\\
 & \leq\sum_{m=1}^{\infty}e^{\theta(p-1)}\int_{m-1}^{m}r^{\theta(p-1)}e^{-r(ph-1)\ln2}dr\\
 & \leq\left(\frac{e}{(ph-1)\ln2}\right)^{\theta(p-1)+1}\int_{0}^{\infty}r^{\theta(p-1)}e^{-r}dr\\
 & \leq\left(\frac{4}{ph-1}\right)^{\theta(p-1)+1}\Gamma[\theta(p-1)+1].
\end{align*}
Now the proof is complete. 
\end{proof}
Using Lemma \ref{lem: supremum p norm}, we show a Doob type inequality
for processes satisfying condition (\ref{eq: conditional increment control}).
To this end, we shall need the following elementary result. 
\begin{lem}
\label{lem: lemma inequalities}Let $\{Y_{t}\}_{t\in[0,T]}$ be a
continuous stochastic process such that $\mathbb{E}(\vert Y_{t}\vert)<\infty$
for all $t\in[0,T]$. Let $0\leq s_{0}<t_{0}\leq T$. Then 
\begin{description}
\item [{(1)}] For any stopping time $\tau$ with $s_{0}\leq\tau\leq t_{0}$,
we have 
\begin{equation}
\left|\mathbb{E}(Y_{t_{0}}\vert\mathcal{F}_{\tau})-Y_{\tau}\right|\leq\mathbb{E}\left[\left.\sup_{u\in[s_{0},t_{0}]}\vert\mathbb{E}(Y_{t_{0}}\vert\mathcal{F}_{u})-Y_{u}\vert\right|\mathcal{F_{\tau}}\right].\label{eq: lemma inequality 1}
\end{equation}
\item [{(2)}] For any $\lambda>0$, we have 
\begin{equation}
\mathbb{P}\left(\sup_{u\in[s_{0},t_{0}]}\vert Y_{u}\vert\geq\lambda\right)\leq\frac{1}{\lambda}\int_{\{\sup_{u\in[s_{0},t_{0}]}\vert Y_{u}\vert\geq\lambda\}}\left[\sup_{u\in[s_{0},t_{0}]}\vert\mathbb{E}(Y_{t_{0}}\vert\mathcal{F}_{u})-Y_{u}\vert+\vert Y_{t_{0}}\vert\right]d\mathbb{P}.\label{eq: Maximal inequality loca inequality}
\end{equation}
\end{description}
\end{lem}

\begin{proof}
(1) By the right continuity of $Y_{t}$ and $\mathbb{E}(Y_{t_{0}}\vert\mathcal{F}_{t})$,
we may assume that $\tau$ takes only countably many values $\{u_{k}:k=1,2,\cdots\}\subset[s_{0},t_{0}].$
Then 
\begin{align*}
\left|\mathbb{E}(Y_{t_{0}}\vert\mathcal{F}_{\tau})-Y_{\tau}\right| & =\sum_{k=1}^{\infty}\left|\mathbb{E}(Y_{t_{0}}\vert\mathcal{F}_{\tau})-Y_{\tau}\right|1_{\{\tau=u_{k}\}}\\
 & =\sum_{k=1}^{\infty}\left|\mathbb{E}\left[\left.(Y_{t_{0}}-Y_{\tau})1_{\{\tau=u_{k}\}}\right|\sigma(\mathcal{F}_{\tau}\cap\{\tau=u_{k}\})\right]\right|\\
 & =\sum_{k=1}^{\infty}\left|\mathbb{E}\left[\left.\mathbb{E}\left((Y_{t_{0}}-Y_{u_{k}})\vert\mathcal{F}_{u_{k}}\right)1_{\{\tau=u_{k}\}}\right|\sigma(\mathcal{F}_{\tau}\cap\{\tau=u_{k}\})\right]\right|\\
 & \leq\sum_{k=1}^{\infty}\left|\mathbb{E}\left[\left.\left(\sup_{u\in[s_{0},t_{0}]}\vert\mathbb{E}(Y_{t_{0}}\vert\mathcal{F}_{u})-Y_{u})\vert\right)1_{\{\tau=u_{k}\}}\right|\sigma(\mathcal{F}_{\tau}\cap\{\tau=u_{k}\})\right]\right|\\
 & =\sum_{k=1}^{\infty}\mathbb{E}\left[\left.\sup_{u\in[s_{0},t_{0}]}\vert\mathbb{E}(Y_{t_{0}}\vert\mathcal{F}_{u})-Y_{u})\vert\right|\mathcal{F}_{\tau}\right]1_{\{\tau=u_{k}\}}\\
 & =\mathbb{E}\left[\left.\sup_{u\in[s_{0},t_{0}]}\vert\mathbb{E}(Y_{t_{0}}\vert\mathcal{F}_{u})-Y_{u})\vert\right|\mathcal{F}_{\tau}\right].
\end{align*}

(2) Let $\tau=\inf\{u\in[s_{0},t_{0}]:\vert Y_{u}\vert\geq\lambda\}\wedge T$.
Then $\{\sup_{u\in[s_{0},t_{0}]}\vert Y_{u}\vert\geq\lambda\}=\{\tau<T\}\cup\{\tau=t_{0},\vert Y_{t_{0}}\vert\geq\lambda\}\in\mathcal{F}_{\tau}$.
Therefore, by (\ref{eq: lemma inequality 1}) we have 
\begin{align*}
\int_{\{\sup_{u\in[s_{0},t_{0}]}\vert Y_{u}\vert\geq\lambda\}}\vert Y_{\tau}\vert d\mathbb{P} & \leq\int_{\{\sup_{u\in[s_{0},t_{0}]}\vert Y_{u}\vert\geq\lambda\}}\vert\mathbb{E}(Y_{t_{0}}\vert\mathcal{F}_{\tau})-Y_{\tau}\vert d\mathbb{P}+\int_{\{\sup_{u\in[s_{0},t_{0}]}\vert Y_{u}\vert\geq\lambda\}}\vert\mathbb{E}(Y_{t_{0}}\vert\mathcal{F}_{\tau})\vert d\mathbb{P}\\
 & \leq\int_{\{\sup_{u\in[s_{0},t_{0}]}\vert Y_{u}\vert\geq\lambda\}}\mathbb{E}\left[\left.\sup_{u\in[s_{0},t_{0}]}\vert\mathbb{E}(Y_{t_{0}}\vert\mathcal{F}_{u})-Y_{u})\vert\right|\mathcal{F}_{\tau}\right]+\mathbb{E}\left(\left.\vert Y_{t_{0}}\vert\right|\mathcal{F}_{\tau}\right)d\mathbb{P}\\
 & \leq\int_{\{\sup_{u\in[s_{0},t_{0}]}\vert Y_{u}\vert\geq\lambda\}}\left[\sup_{u\in[s_{0},t_{0}]}\vert\mathbb{E}(Y_{t_{0}}\vert\mathcal{F}_{u})-Y_{u}\vert+\vert Y_{t_{0}}\vert\right]d\mathbb{P}.
\end{align*}
Finally, using
\[
\mathbb{P}\left(\sup_{u\in[s_{0},t_{0}]}\vert Y_{u}\vert\geq\lambda\right)\leq\frac{1}{\lambda}\int_{\{\sup_{u\in[s_{0},t_{0}]}\vert Y_{u}\vert\geq\lambda\}}\vert Y_{\tau}\vert d\mathbb{P},
\]
and we complete the proof
\end{proof}
Now we are ready to prove Theorem \ref{Theorem: Doob Type Maximal Inequality}.
\begin{proof}
[Proof of Theorem \ref{Theorem: Doob Type Maximal Inequality}]Denote
$Z=\sup_{u\in[s_{0},t_{0}]}\vert\mathbb{E}(X_{t_{0}}\vert\mathcal{F}_{u})-X_{u}\vert+\vert X_{t_{0}}\vert$
and fix a $q\in(1,p]$. By (\ref{eq: Maximal inequality loca inequality})
and Lemma \ref{lem: supremum p norm}

\begin{align*}
\Vert X^{\ast}\Vert_{L^{q}}^{q} & =q\int_{0}^{\infty}\lambda^{q-1}\mathbb{P}(X^{\ast}\geq\lambda)d\lambda\\
 & \leq q\int_{0}^{\infty}\lambda^{q-2}\int_{\{X^{\ast}\geq\lambda\}}Zd\mathbb{P}d\lambda\\
 & =q\int_{\Omega}\left(\int_{0}^{X^{\ast}}\lambda^{q-2}d\lambda\right)Zd\mathbb{P}\\
 & =\frac{q}{q-1}\int_{\Omega}\vert X^{\ast}\vert^{q-1}Zd\mathbb{P}\\
 & \leq\frac{q}{q-1}\Vert X^{\ast}\Vert_{L^{q}}^{q-1}\Vert Z\Vert_{L^{q}}.
\end{align*}
Therefore, 
\begin{align*}
\Vert X^{\ast}\Vert_{L^{q}} & \leq\frac{q}{q-1}\Vert Z\Vert_{L^{q}}\\
 & \leq\frac{q}{q-1}\left[\left\Vert \sup_{u\in[s_{0},t_{0}]}\vert\mathbb{E}(X_{t_{0}}\vert\mathcal{F}_{u})-X_{u}\vert\right\Vert _{L^{q}}+\left\Vert X_{t_{0}}\right\Vert _{L^{q}}\right]\\
 & \leq\frac{q}{q-1}\left[C_{p,h}^{1/P}A_{p,h}^{1/p}\vert t_{0}-s_{0}\vert^{h}+\left\Vert X_{t_{0}}\right\Vert _{L^{q}}\right].
\end{align*}
\end{proof}

\subsection{Tail decay for the supremum}

Here, we shall show that the distribution of the supremum $\sup_{t\in[0,T]}\vert X_{t}\vert$
has $\alpha$-exponential decay if the margins of $\{X_{t}\}_{t\in[0,T]}$
have uniform $\alpha$-exponential decay as in the following definition.
\begin{defn}
Given $\alpha>0$, a continuous stochastic process $\{X_{t}\}_{t\in[0,T]}$
is said to have uniform $\alpha$-exponential marginal decay if there
exists constants $C_{1},C_{2}>0$ such that 
\begin{equation}
\mathbb{P}\left(\vert X_{t}\vert\geq\lambda\right)\leq C_{2}\exp\left(-C_{1}\lambda^{\alpha}\right),\qquad\mbox{ for all }\lambda>0\mbox{ and all }t\in[0,T].\label{eq: alpha exponential decay}
\end{equation}
\end{defn}

\begin{rem*}
Suppose 
\begin{equation}
\mathbb{P}\left(\vert X_{t}\vert\geq\lambda\right)\leq C_{2}\exp\left(-C_{1}\lambda^{\alpha}\right),\qquad\mbox{ for all }\lambda>M\label{eq:alpha exponential decay 2}
\end{equation}
with $M$ being a large enough constant. Since we always have $\mathbb{P}\left(\vert X_{t}\vert\geq\lambda\right)\leq1$,
actually (\ref{eq:alpha exponential decay 2}) implies (\ref{eq: alpha exponential decay})
for another pair of constants $(C_{1},C_{2})$. In the following,
we will always use (\ref{eq: alpha exponential decay}).
\end{rem*}
Notice that if $\{X_{t}\}_{t\in[0,T]}$ satisfies (\ref{eq: alpha exponential decay}),
one has 
\begin{equation}
\mathbb{E}(\vert X_{t}\vert^{q})\leq C_{2}C_{1}^{-q/\alpha}\Gamma(\frac{q}{\alpha}+1)\label{eq: Lq norm control}
\end{equation}
for any $q>0$. Now we state our theorem. 
\begin{thm}
\label{thm: Main theorem 1}Suppose stochastic process $\{X_{t}\}_{t\in[0,T]}$
satisfies condition (\ref{eq: conditional increment control}) with
parameter $(p,h)$ and has uniform $\alpha$-exponential marginal
decay. Then 
\[
\mathbb{P}\left(\sup_{t\in[0,T]}\vert X_{t}\vert\geq\lambda\right)\leq C\lambda^{-1/h}\exp\left[-\frac{C_{1}}{2^{\alpha+2}}\left(1-\frac{1}{ph}\right)\lambda^{\alpha}\right]
\]
for large enough $\lambda$, where $C$ depends on $(C_{1},C_{2},\alpha,p,h,A_{p,h}).$
\end{thm}

\begin{proof}
For $N\in\mathbb{N}_{+}$, let $I_{n}=[t_{n-1},t_{n}]=[(n-1)T/N,nT/N]$,
$1\leq n\leq N$. Then 
\[
\left\{ \sup_{t\in[0,T]}\vert X_{t}\vert\geq2\lambda\right\} \subset\bigcup_{n=1}^{N}\left(\left\{ \sup_{t\in I_{n}}\vert\mathbb{E}(X_{t_{n}}\vert\mathcal{F}_{t})-X_{t}\vert\geq\lambda\right\} \bigcup\left\{ \sup_{t\in I_{n}}\vert\mathbb{E}(X_{t_{n}}\vert\mathcal{F}_{t})\vert\geq\lambda\right\} \right).
\]
Therefore
\begin{align}
\mathbb{P}\left(\sup_{t\in[0,T]}\vert X_{t}\vert\geq2\lambda\right) & \leq\sum_{n=1}^{N}\mathbb{P}\left(\sup_{t\in I_{n}}\vert\mathbb{E}(X_{t_{n}}\vert\mathcal{F}_{t})-X_{t}\vert\geq\lambda\right)+\sum_{n=1}^{N}\mathbb{P}\left(\sup_{t\in I_{n}}\vert\mathbb{E}(X_{t_{n}}\vert\mathcal{F}_{t})\vert\geq\lambda\right).\label{eq: main theorem local 1}
\end{align}
By Lemma \ref{lem: supremum p norm}, 
\begin{equation}
\mathbb{P}\left(\sup_{t\in I_{n}}\vert\mathbb{E}(X_{t_{n}}\vert\mathcal{F}_{t})-X_{t}\vert\geq\lambda\right)\leq C_{p,h}A_{p,h}\lambda^{-p}\left(\frac{T}{N}\right)^{ph}.\label{eq: main theorem local 2}
\end{equation}
Next we need to estimate $\mathbb{P}\left(\sup_{t\in I_{n}}\vert\mathbb{E}(X_{t_{n}}\vert\mathcal{F}_{t})\vert\geq\lambda\right)$.
Notice that 
\[
\mathbb{E}\left[\exp\left(\frac{C_{1}}{4}\sup_{t\in I_{n}}\vert\mathbb{E}(X_{t_{n}}\vert\mathcal{F}_{t})\vert^{\alpha}\right)\right]=\sum_{q=0}^{\infty}\frac{(C_{1}/4)^{q}}{q!}\mathbb{E}\left(\sup_{t\in I_{n}}\vert\mathbb{E}(X_{t_{n}}\vert\mathcal{F}_{t})\vert^{\alpha q}\right).
\]
Here we fix an arbitrary constant $\beta>1$. When $\alpha q\leq\beta$,
by Doob's maximal inequality and (\ref{eq: Lq norm control}), we
have that 
\begin{align*}
\mathbb{E}\left(\sup_{t\in I_{n}}\vert\mathbb{E}(X_{t_{n}}\vert\mathcal{F}_{t})\vert^{\alpha q}\right) & \leq\mathbb{E}\left(\sup_{t\in I_{n}}\vert\mathbb{E}(X_{t_{n}}\vert\mathcal{F}_{t})\vert^{\beta}\right)^{\frac{\alpha q}{\beta}}\\
 & \leq\left(\frac{\beta}{\beta-1}\right)^{\beta\cdot\frac{\alpha q}{\beta}}\mathbb{E}\left(\vert X_{t_{n}}\vert^{\beta}\right)^{\frac{\alpha q}{\beta}}\\
 & \leq\left(\frac{\beta}{\beta-1}\right)^{\alpha q}\left(C_{2}C_{1}^{-\beta/\alpha}\Gamma(\frac{\beta}{\alpha}+1)\right)^{\frac{\alpha q}{\beta}}.
\end{align*}
When $\alpha q>\beta$, again by Doob's maximal inequality and (\ref{eq: Lq norm control}),
we have that 
\begin{align*}
\mathbb{E}\left(\sup_{t\in I_{n}}\vert\mathbb{E}(X_{t_{n}}\vert\mathcal{F}_{t})\vert^{\alpha q}\right) & \leq\mathbb{E}\left(\sup_{t\in I_{n}}\vert\mathbb{E}(\vert X_{t_{n}}\vert^{\frac{\alpha q}{\beta}}\vert\mathcal{F}_{t})\vert^{\beta}\right)\\
 & \leq\left(\frac{\beta}{\beta-1}\right)^{\beta}\mathbb{E}\left(\vert X_{t_{n}}\vert^{\alpha q}\right)\\
 & \leq\left(\frac{\beta}{\beta-1}\right)^{\beta}C_{2}C_{1}^{-q}\Gamma(q+1).
\end{align*}
Therefore
\begin{align*}
\mathbb{E}\left[\exp\left(\frac{C_{1}}{4}\sup_{t\in I_{n}}\vert\mathbb{E}(X_{t_{n}}\vert\mathcal{F}_{t})\vert^{\alpha}\right)\right] & \leq\sum_{q=0}^{\lfloor\frac{\beta}{\alpha}\rfloor}\frac{(C_{1}/4)^{q}}{q!}\left(\frac{\beta}{\beta-1}\right)^{\alpha q}\left(C_{2}C_{1}^{-\beta/\alpha}\Gamma(\frac{\beta}{\alpha}+1)\right)^{\frac{\alpha q}{\beta}}\\
 & \quad+\sum_{q=\lfloor\frac{\beta}{\alpha}\rfloor+1}^{\infty}\frac{(C_{1}/4)^{q}}{q!}\left(\frac{\beta}{\beta-1}\right)^{\beta}C_{2}C_{1}^{-q}\Gamma(q+1)\\
 & \leq C+\sum_{q=\lfloor\frac{\beta}{\alpha}\rfloor+1}^{\infty}4^{-q}\left(\frac{\beta}{\beta-1}\right)^{\beta}C_{2}\\
 & \leq C,
\end{align*}
where the constant C depends on $(\alpha,\beta,C_{1},C_{2})$. By
Chebyshev's inequality, 
\begin{equation}
\mathbb{P}\left(\sup_{t\in I_{n}}\vert\mathbb{E}(X_{t_{n}}\vert\mathcal{F}_{t})\vert\geq\lambda\right)\leq C\exp\left(-\frac{C_{1}}{4}\lambda^{\alpha}\right).\label{eq: main theorem local 3}
\end{equation}
Hence, for any $N\in\mathbb{N}_{+}$, by (\ref{eq: main theorem local 2})
and (\ref{eq: main theorem local 3}), we have that
\begin{align*}
\mathbb{P}\left(\sup_{t\in[0,T]}\vert X_{t}\vert\geq2\lambda\right) & \leq NC_{p,h}A_{p,h}\lambda^{-p}\left(\frac{T}{N}\right)^{ph}+NC\exp\left(-\frac{C_{1}}{4}\lambda^{\alpha}\right)\\
 & =:EN^{1-ph}+FN,
\end{align*}
where $E=C_{p,h}A_{p,h}T^{ph}\lambda^{-p}$ and $F=C\exp\left(-\frac{C_{1}}{4}\lambda^{\alpha}\right)$.
When $\lambda$ is large enough, notice that $E\gg F$ and $E/F$
is large enough. Then we can set $N$ to be the greatest integer less
than or equal to $\left(\frac{E}{F}\right)^{\frac{1}{ph}}$ to obtain
that
\begin{align*}
\mathbb{P}\left(\sup_{t\in[0,T]}\vert X_{t}\vert\geq2\lambda\right) & \leq CE^{\frac{1}{ph}}F^{1-\frac{1}{ph}}\\
 & =CC_{p,h}^{\frac{1}{ph}}A_{p,h}^{\frac{1}{ph}}T\lambda^{-\frac{1}{h}}\exp\left(-\frac{C_{1}}{4}\left(1-\frac{1}{ph}\right)\lambda^{\alpha}\right),
\end{align*}
where $C$ depends on $C_{1}$, $C_{2}$, $p$, $h$ and $\alpha$.
Finally, replace $2\lambda$ by $\lambda$ and the proof is complete.
\end{proof}

\section{SDEs with singular drift\label{sec:SDEs}}

In this section, we apply the results in the previous section to solve
SDEs 
\begin{equation}
dX_{t}=b(t,X_{t})dt+dB_{t},\label{eq: SDE}
\end{equation}
where divergence-free vector field $b\in L^{l}(0,T;L^{q}(\rd))\cap L^{1}(0,T;W^{1,p}(\rd))$
with $\frac{2}{l}+\frac{d}{q}=\gamma\in[1,2)$, $d\geq3$ and $p\geq1$.
We use the idea of approximation and \textit{a priori} estimates to
show that the approximation sequence converges in probability. Together
with the $L^{k}(\Omega\times B_{r};C([0,T]))$ bound of the approximation
sequence, we have that the sequence converges in $L^{k}(\Omega\times B_{r};C([0,T]))$
for any $k\geq1$. Here $B_{r}$ is the ball in $\rd$ of radius $r$
and center at the origin.

\subsection{Supremum of solutions to SDEs}

To control the supremum $\sup_{t\in[0,T]}\vert X_{t}\vert$, we will
need the following Aronson type estimate of the transition probability
of $X_{t}$ proved in \cite[Corollary 9]{qian2017parabolic}.
\begin{thm}
Suppose $b$ is divergence-free and $b\in L^{l}(0,T;L^{q}(\rd))$
for some $d\geq3$, $l>1$ and $q>\frac{d}{2}$ such that $\frac{2}{l}+\frac{d}{q}=\gamma\in[1,2)$.
In addition, we assume that $b$ is smooth with bounded derivatives.
If $\mu:=\frac{2}{2-\gamma+\frac{2}{l}}>1$, the transition probability
has upper bound 
\[
\Gamma(t,x;\tau,\xi)\leq\begin{cases}
\frac{C_{1}}{(t-\tau)^{d/2}}\exp\left(-\frac{1}{C_{2}}\left(\frac{\vert x-\xi\vert^{2}}{t-\tau}\right)\right) & \frac{\vert x-\xi\vert^{\mu-2}}{(t-\tau)^{\mu-\nu-1}}<1\\
\frac{C_{1}}{(t-\tau)^{d/2}}\exp\left(-\frac{1}{C_{2}}\left(\frac{\vert x-\xi\vert^{\mu}}{(t-\tau)^{\nu}}\right)^{\frac{1}{\mu-1}}\right) & \frac{\vert x-\xi\vert^{\mu-2}}{(t-\tau)^{\mu-\nu-1}}\geq1,
\end{cases}
\]
where $\nu=\frac{2-\gamma}{2-\gamma-\frac{2}{l}}$, $\Lambda=\Vert b\Vert_{L^{l}(0,T;L^{q}(\rd))}$,
$C_{1}=C_{1}(l,q,d)$, $C_{2}=C_{2}(l,q,d,\Lambda)$. If $\mu=1$,
which implies $q=\infty$, we have 
\[
\Gamma(t,x;\tau,\xi)\leq\frac{C_{1}}{(t-\tau)^{d/2}}\exp\left(-\frac{(C_{1}\Lambda(t-\tau)^{\nu}-\vert x-\xi\vert)^{2}}{4C_{1}(t-\tau)}\right).
\]
\end{thm}

\begin{rem}
This theorem is actually true for diffusion processes corresponding
to parabolic equations 
\[
\partial_{t}u(t,x)-\sum_{i,j=1}^{d}\partial_{j}(a_{ij}(t,x)\partial_{i}u(t,x))+\sum_{i=1}^{d}b_{i}(t,x)\partial_{i}u(t,x)=0
\]
for uniformly elliptic $\{a_{ij}\}$, $b\in L^{l}(0,T;L^{q}(\rd))$
with $\frac{2}{l}+\frac{d}{q}=\gamma\in[1,2)$ and $\divg b\in L^{l'}(0,T;L^{q'}(\rd))$
with $\frac{2}{l'}+\frac{d}{q'}=\gamma'\in[1,2)$. The proof follows
the idea in \cite{qian2017parabolic} with small modification.
\end{rem}

This upper bound estimate of the transition probability implies the
following. 
\begin{prop}
\label{prop: lp norm and exponential decay}Suppose $b$ is a smooth
divergence-free vector field with bounded derivatives and $b\in L^{l}(0,T;L^{q}(\rd))$
for some $d\geq3$, $l>1$ and $q>\frac{d}{2}$ such that $1\leq\gamma<2$.
Then the solution $\{X_{t}\}_{t\in[0,T]}$ to (\ref{eq: SDE}) satisfies
that 
\begin{equation}
\mathbb{E}\vert X_{t}-X_{s}\vert^{p}\leq C\vert t-s\vert^{\frac{(2-\gamma)(p+d)-d}{2}}\label{eq: moment estimate}
\end{equation}
 for $0\leq s<t\leq T$. Moreover, there exists a constant $\alpha>1$
depending on $(l,q,d)$ such that 
\begin{equation}
\mathbb{P}(\vert X_{t}-X_{s}\vert>\lambda)\leq C\exp(-C\lambda^{\alpha})\label{eq: exponential decay}
\end{equation}
for large enough $\lambda$ and $0\leq s<t\leq T$. The constant $C$
depends on $(l,q,d,\Lambda)$.
\end{prop}

\begin{proof}
Without loss of generality, we may take $s=0$ and $X_{0}=0$. When
$\mu>1$, we have 
\begin{align*}
\mathbb{E}\vert X_{t}-X_{0}\vert^{p} & =\int_{\rd}\vert x\vert^{p}\Gamma(t,x;0,0)dx\\
 & \leq\int_{\rd}\vert x\vert^{p}\frac{C_{1}}{t^{d/2}}\exp\left(-\frac{1}{C_{2}}\left(\frac{\vert x\vert^{2}}{t}\right)\right)dx+\int_{\rd}\vert x\vert^{p}\frac{C_{1}}{t^{d/2}}\exp\left(-\frac{1}{C_{2}}\left(\frac{\vert x\vert^{\mu}}{t^{\nu}}\right)^{\frac{1}{\mu-1}}\right)dx\\
 & \leq C\left(t^{\frac{p}{2}}+t^{\frac{(2-\gamma)(p+d)-d}{2}}\right)\\
 & \leq Ct^{\frac{(2-\gamma)(p+d)-d}{2}}.
\end{align*}
When $\mu=1$, following similar argument, we have that 
\begin{align*}
\mathbb{E}\vert X_{t}-X_{0}\vert^{p} & =\int_{\rd}\vert x\vert^{p}\Gamma(t,x;0,0)dx\\
 & \leq\int_{\vert x\vert\geq Ct^{\nu}}\vert x\vert^{p}\Gamma(t,x;0,0)dx+\int_{\vert x\vert<Ct^{\nu}}\vert x\vert^{p}\Gamma(t,x;0,0)dx\\
 & \leq\int_{\rd}\vert x\vert^{p}\frac{C_{1}}{t^{d/2}}\exp\left(-\frac{C\vert x\vert^{2}}{4C_{1}t}\right)dx+\int_{\vert x\vert<Ct^{\nu}}\vert x\vert^{p}\frac{C_{1}}{t^{d/2}}dx\\
 & \leq C\left(t^{\frac{p}{2}}+t^{\frac{(2-\gamma)(p+d)-d}{2}}\right)\\
 & \leq Ct^{\frac{(2-\gamma)(p+d)-d}{2}}.
\end{align*}
Now we prove the uniform $\alpha$-exponential marginal decay for
$X_{t}$. When $\mu>1$, we have that 
\begin{align*}
\mathbb{P}(\vert X_{t}\vert>\lambda) & \leq\int_{\vert x\vert>\lambda}\frac{C_{1}}{t^{d/2}}\exp\left(-\frac{1}{C_{2}}\left(\frac{\vert x\vert^{2}}{t}\right)\right)dx+\int_{\vert x\vert>\lambda}\frac{C_{1}}{t^{d/2}}\exp\left(-\frac{1}{C_{2}}\left(\frac{\vert x\vert^{\mu}}{t^{\nu}}\right)^{\frac{1}{\mu-1}}\right)dx\\
 & =C_{1}\int_{\vert x\vert>\lambda t^{-\frac{1}{2}}}\exp\left(-\frac{\vert x\vert^{2}}{C_{2}}\right)dx+\frac{C_{1}}{t^{d(\gamma-1)/2}}\int_{\vert x\vert>\lambda t^{-\frac{2-\gamma}{2}}}\exp\left(-\frac{\vert x\vert^{\frac{\mu}{\mu-1}}}{C_{2}}\right)dx\\
 & \leq C\exp(-C(\lambda t^{-\frac{1}{2}})^{2})+\frac{C}{t^{d(\gamma-1)/2}}\exp\left(-C(\lambda t^{-\frac{2-\gamma}{2}})^{\frac{\mu}{\mu-1}}\right)\\
 & \leq C\exp(-C\lambda^{2})+C\exp(-C\lambda^{\frac{\mu}{\mu-1}})
\end{align*}
for $0<t\leq T$ and large enough $\lambda$. Similarly, when $\mu=1$,
we have 
\begin{align*}
\mathbb{P}(\vert X_{t}\vert>\lambda) & \leq\int_{\vert x\vert>\lambda}\frac{C_{1}}{t^{d/2}}\exp\left(-\frac{(C_{1}\Lambda t^{\nu}-\vert x\vert)^{2}}{4C_{1}t}\right)dx\\
 & \leq C_{1}\int_{\vert x\vert>\lambda t^{-\frac{1}{2}}}\exp\left(-C(Ct^{\nu-\frac{1}{2}}-\vert x\vert)^{2}\right)dx.
\end{align*}
Recall that $\mu=1$ implies $q=\infty$ and hence $\nu=\frac{2-\gamma}{2}\in(0,\frac{1}{2}]$.
For large enough $\lambda$, we have $\vert x\vert>\lambda t^{-\frac{1}{2}}\gg Ct^{\nu-\frac{1}{2}}$
and
\begin{align*}
\mathbb{P}(\vert X_{t}\vert>\lambda) & \leq C_{1}\int_{\vert x\vert>\lambda t^{-\frac{1}{2}}}\exp\left(-C\vert x\vert^{2}\right)\\
 & \leq C\exp(-C\lambda^{2}).
\end{align*}
\end{proof}
Now we can apply Theorem \ref{thm: Main theorem 1} to obtain the
following result.
\begin{prop}
\label{prop: supremum estimate of diffusion}Suppose $b$ is a smooth
divergence-free vector field with bounded derivatives and $b\in L^{l}(0,T;L^{q}(\rd))$
for some $d\geq3$, $l>1$ and $q>\frac{d}{2}$ such that $1\leq\gamma<2$.
Then the solution $\{X_{t}\}_{t\in[0,T]}$ to (\ref{eq: SDE}) satisfies
that 
\begin{equation}
\mathbb{P}\left(\sup_{t\in[0,T]}\vert X_{t}-X_{0}\vert>\lambda\right)\leq C\exp(-C\lambda^{\alpha})\label{eq: exponential decay supremum}
\end{equation}
with the same $\alpha$ as in Proposition \ref{prop: lp norm and exponential decay}.
Moreover, for any $p\geq1$ we have that 
\begin{equation}
\mathbb{E}\left[\sup_{t\in[0,T]}\vert X_{t}-X_{0}\vert^{p}\right]<C,\label{eq: moment estimate supremum}
\end{equation}
where $C$ depends on $l$, $q$, $p$, $d$ and $\Vert b\Vert_{L^{l}(0,T;L^{q}(\rd))}$.
\end{prop}

\begin{rem*}
Using Kolmogorov's continuity theorem, for any large enough $p$ such
that $\frac{(2-\gamma)(p+d)-d}{2}>1$, inequality (\ref{eq: moment estimate})
implies that there exists $\alpha>0$ such that 
\[
\vert X_{t}-X_{s}\vert\leq K\vert t-s\vert^{\alpha}\qquad\mbox{for all }0\leq s<t\leq T.
\]
Here $K$ is a random variable satisfying $\mathbb{E}[K^{p}]<C$.
Since this is true for any large enough $p$, in this way we can also
prove the moment estimate (\ref{eq: moment estimate supremum}) for
any $p\geq1$, although it can not prove the exponential decay (\ref{eq: exponential decay supremum}).
But actually Theorem \ref{thm: SDE main theorem} below uses only
the moment estimate (\ref{eq: moment estimate supremum}). This provides
an alternative proof to Theorem \ref{thm: SDE main theorem}.
\end{rem*}

\subsection{Existence and uniqueness of strong solutions}

Now we can construct a unique almost everywhere stochastic flow to
(\ref{eq: SDE}) using approximation. The argument essentially follows
Zhang \cite{ZhangXicheng2009}. Firstly, we define almost everywhere
stochastic flow as follows.
\begin{defn}
\label{def: almost everywhere stochastic flow}Suppose $\{X_{t}\}_{t\in[0,T]}$
is a $\rd$-valued stochastic process defined on $\Omega\times\rd\times[0,T]$.
We say that $\{X_{t}\}_{t\in[0,T]}$ is an almost everywhere stochastic
flow to (\ref{eq: SDE}) if
\end{defn}

\begin{description}
\item [{(1)}] for $\mathbb{P}\times m$-almost all $(\omega,x)\in\Omega\times\rd$,
$t\rightarrow X_{t}(\omega,x)$ is a $\rd$-continuous function on
$[0,T]$;
\item [{(2)}] for $\mathbb{P}$-almost all $\omega\in\Omega$, under mapping
$x\rightarrow X_{t}(\omega,x)$, the push-forward of the Lebesgue
measure $m$ restricted to any Borel set $A\subset\rd$ has density,
i.e. $(m1_{A})\circ X_{t}^{-1}(\omega)=\rho_{t}(\omega,A,x)dx$, where
the density satisfies that $\rho_{t}(\omega,A,\cdot)\leq1$ for all
$x\in\rd$ and $\int_{\rd}\rho_{t}(\omega,A,x)dx=m(A)$;
\item [{(3)}] for any $t\in[0,T]$, we have
\[
X_{t}(x)=x+\int_{0}^{t}b(s,X_{s}(x))ds+\int_{0}^{t}dB_{s}
\]
for $\mathbb{P}\times m$-almost all $(\omega,x)\in\Omega\times\rd$.
\end{description}
When the vector field $b$ is smooth and divergence-free, the strong
solution $X_{t}$ to (\ref{eq: SDE}) preserves the Lebesgue measure
in the sense that 
\begin{equation}
\mathbb{P}\left[\omega\in\Omega:m(X_{t}(\omega,A))=m(A)\right]=1,\label{eq: preserve lebesgue measure}
\end{equation}
where $X_{t}(\omega,A)$ is the image of any Borel set $A\in\rd$
under mapping $x\rightarrow X_{t}(\omega,x)$. Clearly $X_{t}$ satisfies
Definition \ref{def: almost everywhere stochastic flow}.

We first recall the following lemma in Crippa and De Lellis \cite[Lemma A.3]{DelellisCrippa2008}.
\begin{lem}
\label{lem: local maximal estimate}Let $M_{R}f$ be the local maximal
function of locally integrable function $f$ defined as 
\[
M_{R}f(x)=\sup_{0<r<R}\frac{1}{\vert B_{r}\vert}\int_{B_{r}(x)}f(y)dy.
\]
Suppose $f\in BV_{loc}(\rd)$, then 
\begin{equation}
\vert f(x)-f(y)\vert\leq C\vert x-y\vert\left[M_{R}\vert\nabla f\vert(x)-M_{R}\vert\nabla f\vert(y)\right]\label{eq: local maximal inequality 1}
\end{equation}
for $x,y\in\rd\backslash N$, where $N$ is a negligible set in $\rd$,
$R=\vert x-y\vert$ is the distance between $x$ and $y$, and constant
$C$ depends only on the dimension $d$. 
\end{lem}

We denote by $Mf$ the maximal function 
\[
Mf(x)=\sup_{0<r<\infty}\frac{1}{\vert B_{r}\vert}\int_{B_{r}(x)}f(y)dy
\]
and clearly inequality (\ref{eq: local maximal inequality 1}) is
also true if we replace $M_{R}\vert\nabla f\vert$ with $M\vert\nabla f\vert$.
\begin{lem}
\label{lem: the log estiamte}Suppose $X_{t}(x)$ and $\tilde{X}_{t}(x)$
are almost everywhere stochastic flows to SDE (\ref{eq: SDE}) driven
by the same Brownian motion, with initial data $x$ and drifts $b$
and $\tilde{b}$ in $L^{1}(0,T;W^{1,p}(\rd))$, $p\geq1$ respectively.
Then for any $r>0$ and $\theta>0$,
\[
\mathbb{E}\left[\int_{B_{r}}\log\left(\frac{\sup_{0\leq t\leq T}\vert X_{t}(x)-\tilde{X}_{t}(x)\vert^{2}}{\theta^{2}}+1\right)dx\right]\leq C\left(\Vert\nabla b\Vert_{L^{1}(0,T;L^{p}(\rd))}+\frac{1}{\theta}\Vert b-\tilde{b}\Vert_{L^{1}(0,T;L^{p}(\rd))}\right),
\]
where the constant $C$ depends on $(r,p,d)$.
\end{lem}

\begin{proof}
Consider 
\begin{align*}
\frac{d}{dt}\log\left(\frac{\vert X_{t}(x)-\tilde{X}_{t}(x)\vert^{2}}{\theta^{2}}+1\right) & \leq\frac{\vert X_{t}(x)-\tilde{X}_{t}(x)\vert\vert b(t,X_{t}(x))-\tilde{b}(t,\tilde{X}_{t}(x))\vert}{\vert X_{t}(x)-\tilde{X}_{t}(x)\vert^{2}+\theta^{2}}\\
 & \leq\frac{\vert b(t,X_{t}(x))-b(t,\tilde{X}_{t}(x))\vert}{\sqrt{\vert X_{t}(x)-\tilde{X}_{t}(x)\vert^{2}+\theta^{2}}}+\frac{\vert b(t,\tilde{X}_{t}(x))-\tilde{b}(t,\tilde{X}_{t}(x))\vert}{\sqrt{\vert X_{t}(x)-\tilde{X}_{t}(x)\vert^{2}+\theta^{2}}}\\
 & =g_{1}(x)+g_{2}(x).
\end{align*}
Integrate both sides on $B_{r}$ and take expectation, then by Lemma
\ref{lem: local maximal estimate} we have that 
\begin{align*}
\mathbb{E}\left[\int_{B_{r}}g_{1}(x)dx\right] & \leq\mathbb{E}\left[\int_{B_{r}}\frac{C\vert X_{t}(x)-\tilde{X}_{t}(x)\vert(M\vert\nabla b\vert(t,X_{t}(x))+M\vert\nabla b\vert(t,\tilde{X}_{t}(x))}{\sqrt{\vert X_{t}(x)-\tilde{X}_{t}(x)\vert^{2}+\theta^{2}}}dx\right]\\
 & \leq C\int_{\Omega}\left(\int_{B_{r}}M\vert\nabla b\vert(t,X_{t}(\omega,x))dx+\int_{B_{r}}M\vert\nabla b\vert(t,\tilde{X}_{t}(\omega,x))dx\right)d\mathbb{P}(\omega).\\
 & =C\int_{\Omega}\left(\int_{\rd}M\vert\nabla b\vert(t,x)\rho_{t}(\omega,B_{r},x)dx+\int_{\rd}M\vert\nabla b\vert(t,x)\tilde{\rho}_{t}(\omega,B_{r},x)dx\right)d\mathbb{P}(\omega).
\end{align*}
Here $\Vert M\vert\nabla b\vert\Vert_{L^{p}(\rd)}\leq C\Vert\nabla b\Vert_{L^{p}(\rd)}$
and $L^{p}(\rd)\subset L^{1}(\rd)+L^{\infty}(\rd)$. For any $f\in L^{p}(\rd)$,
we have $f=f_{1}+f_{2}$, where $f_{1}=f1_{\{f<\Vert f\Vert_{L^{p}}\}}$
and $f_{2}=f1_{\{f\geq\Vert f\Vert_{L^{p}}\}}$. It is easy to verify
that $\Vert f_{1}\Vert_{L^{\infty}(\rd)}\leq\Vert f\Vert_{L^{p}(\rd)}$
and $\Vert f_{2}\Vert_{L^{1}(\rd)}\leq\Vert f\Vert_{L^{p}(\rd)}$.
By (2) in Definition \ref{def: almost everywhere stochastic flow},
we have that $\Vert\rho_{t}(\omega,B_{r},\cdot)\Vert_{L^{\infty}(\rd)}\leq1$,
$\Vert\rho_{t}(\omega,B_{r},\cdot)\Vert_{L^{1}(\rd)}\leq\vert B_{r}\vert$
and the same is true for $\tilde{\rho}_{t}(\omega,B_{r}\cdot)$. Hence
\begin{align*}
\mathbb{E}\left[\int_{B_{r}}g_{1}(x)dx\right] & \leq2C\int_{\Omega}(1+\vert B_{r}\vert)\Vert\nabla b\Vert_{L^{p}(\rd)}d\mathbb{P}(\omega)\\
 & =2C(1+\vert B_{r}\vert)\Vert\nabla b\Vert_{L^{p}(\rd)}.
\end{align*}
Similarly, we have 
\begin{align*}
\mathbb{E}\left[\int_{B_{r}}g_{2}(x)dx\right] & \leq\frac{1}{\theta}\mathbb{E}\left[\int_{B_{r}}\vert b(t,\tilde{X}_{t}(x))-\tilde{b}(t,\tilde{X}_{t}(x))\vert dx\right]\\
 & \leq\frac{1}{\theta}\int_{\Omega}\int_{\rd}\vert b-\tilde{b}\vert(t,x)\tilde{\rho}_{t}(\omega,B_{r},x)dxd\mathbb{P}(\omega)\\
 & \leq\frac{1}{\theta}(1+\vert B_{r}\vert)\Vert b-\tilde{b}\Vert_{L^{p}(\rd)}.
\end{align*}
Finally, we integrate in $t$, and take supremum over time $t$ for
$\log\left(\frac{\vert X_{t}(x)-\tilde{X}_{t}(x)\vert^{2}}{\theta^{2}}+1\right)$
and the proof is complete.
\end{proof}
Now we are ready to prove the main result in this section. 
\begin{thm}
\label{thm: SDE main theorem}Given a divergence-free vector field
$b\in L^{1}(0,T;W^{1,p}(\rd))\cap L^{l}(0,T;L^{q}(\rd))$ with $d\geq3$,
$p\geq1$, $\frac{2}{l}+\frac{d}{q}\in[1,2)$, there is a unique almost
everywhere stochastic flow $X(\omega,x):\Omega\times\rd\rightarrow C([0,T],\rd)$
to
\[
dX_{t}(\omega,x)=b(t,X_{t}(\omega,x))dt+dB_{t}(\omega),\qquad X_{0}(\omega,x)=x
\]
in space $L^{k}(\Omega\times B_{r};C([0,T],\rd))$ for any $k\geq1$
and $r>0$.
\end{thm}

\begin{proof}
Step 1: We prove the existence of solution $X_{t}$ using approximation.
By cut-off and mollification, we can find a sequence of divergence-free
$b^{(n)}\in C([0,T],C_{0}^{\infty}(\rd))$ such that $b^{(n)}\rightarrow b$
in $L^{1}(0,T;W^{1,p}(\rd))\cap L^{l}(0,T;L^{q}(\rd))$ and denote
by $X_{t}^{(n)}$ the corresponding solution. We first prove that
$X_{t}^{(n)}$ is a Cauchy sequence in space $L^{k}(\Omega\times B_{r};C([0,T]))$
for any $k\geq1$ and $r>0$. Denote 
\[
O_{n,m}^{R}(\omega)=\left\{ x\in\rd:\sup_{0\leq t\leq T}\vert X_{t}^{(n)}(\omega,x)\vert<R,\sup_{0\leq t\leq T}\vert X_{t}^{(m)}(\omega,x)\vert<R\right\} 
\]
and by Proposition \ref{prop: supremum estimate of diffusion} we
have that for any fixed $r>0$, 
\begin{equation}
\lim_{R\rightarrow\infty}\sup_{n,m}\sup_{x\in B_{r}}\mathbb{P}(\omega:x\notin O_{n,m}^{R}(\omega))=0.\label{eq: theorem local inequality 1}
\end{equation}
Set $S_{T}^{(n,m)}(\omega,x)=\sup_{0\leq t\leq T}\vert X_{t}^{(n)}(\omega,x)-X_{t}^{(m)}(\omega,x)\vert^{2}$,
then for any fixed $\delta>0$ we have 
\begin{align*}
\mathbb{P}\left(\omega:\int_{B_{r}}S_{T}^{(n,m)}(\omega,x)dx\geq2\delta\right) & \leq\mathbb{P}\left(\omega:\int_{B_{r}\cap O_{n,m}^{R}(\omega)}S_{T}^{(n,m)}(\omega,x)dx\geq\delta\right)\\
 & \quad+\mathbb{P}\left(\omega:\int_{B_{r}\backslash O_{n,m}^{R}(\omega)}S_{T}^{(n,m)}(\omega,x)dx\geq\delta\right)\\
 & =I_{1}^{(n,m)}+I_{2}^{(n,m)}.
\end{align*}
To show convergence in probability of $X_{t}^{(n)}$, for any $\epsilon>0$,
we find $R$ and large enough $n,m$ such that $I_{i}^{n,m}\leq\epsilon$,
$i=1,2$. We first estimate the second term $I_{2}^{(n,m)}$
\begin{align*}
I_{2}^{(n,m)} & \leq\frac{1}{\delta}\mathbb{E}\left[\int_{B_{r}\backslash O_{n,m}^{R}(\omega)}\sup_{0\leq t\leq T}\vert X_{t}^{(n)}(\omega,x)-X_{t}^{(m)}(\omega,x)\vert^{2}dx\right]\\
 & \leq\frac{1}{\delta}\int_{B_{r}}\int_{\Omega}\sup_{0\leq t\leq T}\vert X_{t}^{(n)}(\omega,x)-X_{t}^{(m)}(\omega,x)\vert^{2}1_{\{\omega:x\notin O_{n,m}^{R}(\omega)\}}d\mathbb{P}(\omega)dx\\
 & \leq\frac{1}{\delta}\int_{B_{r}}4\mathbb{E}\left[\sup_{0\leq t\leq T}\vert X_{t}^{(n)}(\omega,x)-x\vert^{4}+\sup_{0\leq s\leq t}\vert X_{t}^{(m)}(\omega,x)-x\vert^{4}\right]^{\frac{1}{2}}\mathbb{P}(\omega:x\notin O_{n,m}^{R}(\omega))^{\frac{1}{2}}dx\\
 & \leq\epsilon
\end{align*}
for large enough $R>M$ by (\ref{eq: theorem local inequality 1})
and Proposition \ref{prop: supremum estimate of diffusion}. Here
the choice of $M$ is independent of $(n,m)$. To obtain the estimate
that $I_{2}^{(n,m)}\leq\epsilon$, we fixed an $R$. With the same
$R$, next we estimate $I_{1}^{(n,m)}$. For any $\omega\in\Omega$,
if
\[
\int_{B_{r}}\log\left(\frac{S_{T}^{(n,m)}(\omega,x)}{\theta^{2}}+1\right)dx\leq L,
\]
we have that $\vert\{S_{T}^{(n,m)}(\omega,x)\geq\theta^{2}(e^{L^{2}}-1)\}\vert\leq\frac{1}{L}$,
which implies 
\begin{align*}
\int_{B_{r}\cap O_{n,m}^{R}(\omega)}S_{T}^{n,m}(\omega,x)dx & =\int_{B_{r}\cap O_{n,m}^{R}(\omega)}S_{T}^{(n,m)}(\omega,x)1_{\{S_{T}^{(n,m)}(x)\geq\theta^{2}(e^{L^{2}}-1)\}}dx\\
 & \quad+\int_{B_{r}\cap O_{n,m}^{R}(\omega)}S_{T}^{(n,m)}(\omega,x)1_{\{S_{T}^{(n,m)}(x)<\theta^{2}(e^{L^{2}}-1)\}}dx\\
 & \leq\theta^{2}(e^{L^{2}}-1)\vert B_{r}\vert+4R^{2}\frac{1}{L}.
\end{align*}
Now we set $\theta^{(n,m)}=\Vert b^{(n)}-b^{(m)}\Vert_{L_{t}^{1}L_{x}^{p}}$
to obtain 
\[
\sup_{n,m}\mathbb{E}\left[\int_{B_{r}}\log\left(\frac{S_{T}^{(n,m)}}{(\theta^{(n,m)})^{2}}+1\right)dx\right]\leq C
\]
by Lemma \ref{lem: the log estiamte}, which implies that 
\[
\sup_{n,m}\mathbb{P}\left(\int_{B_{r}}\log\left(\frac{S_{T}^{(n,m)}}{(\theta^{(n,m)})^{2}}+1\right)dx\geq L\right)\leq\frac{C}{L}.
\]
For fixed $\delta>0$ and the fixed $R$ obtained from the estimate
of $I_{2}^{(n,m)}$, we can first choose $L$ large enough and them
choose $(n,m)$ large enough, which means $\theta^{(n,m)}$ is small
enough, such that 
\[
(\theta^{(n,m)})^{2}(e^{L^{2}}-1)\vert B_{r}\vert+4R^{2}\frac{1}{L}<\delta\quad\mbox{and}\quad\frac{C}{L}\leq\epsilon.
\]
Hence 
\[
\mathbb{P}\left(\omega:\int_{B_{r}\cap O_{n,m}^{R}(\omega)}S_{T}^{n,m}(\omega,x)dx\geq\delta,\int_{B_{r}}\log\left(\frac{S_{T}^{(n,m)}(\omega,x)}{(\theta^{(n,m)})^{2}}+1\right)dx\leq L\right)=0,
\]
which implies that 
\begin{align*}
I_{1}^{(n,m)} & =\mathbb{P}\left(\omega:\int_{B_{r}\cap O_{n,m}^{R}(\omega)}S_{T}^{n,m}(\omega,x)dx\geq\delta,\int_{B_{r}}\log\left(\frac{S_{T}^{(n,m)}(\omega,x)}{(\theta^{(n,m)})^{2}}+1\right)dx>L\right)\\
 & \leq\mathbb{P}\left(\omega:\int_{B_{r}}\log\left(\frac{S_{T}^{(n,m)}(\omega,x)}{(\theta^{(n,m)})^{2}}+1\right)dx>L\right)\\
 & \leq\epsilon.
\end{align*}
Now we have that for any $\epsilon>0$, there is $(n,m)$ large enough
such that 
\[
\mathbb{P}\left(\omega:\int_{B_{r}}S_{T}^{(n,m)}(\omega,x)dx\geq2\delta\right)\leq2\epsilon.
\]
Hence 
\[
\lim_{n,m\rightarrow\infty}\mathbb{P}\left(\omega:\int_{B_{r}}S_{T}^{(n,m)}(\omega,x)dx\geq2\delta\right)=0
\]
for any fixed $\delta$. This implies that for any fixed $r>0$, $\{X_{t}^{(n)}\}$
converges in probability under the finite measure $\mathbb{P}\times m1_{B_{r}}$
as functions $X^{(n)}:\Omega\times B_{r}\rightarrow C([0,T],\rd)$.
Recall that for any $k\geq1$ we have
\[
\sup_{n,x\in B_{r}}\mathbb{E}\left[\sup_{0\leq s\leq t}\vert X_{s}^{(n)}(x)\vert^{k}\right]<\infty
\]
by Proposition \ref{prop: supremum estimate of diffusion}, which
means that for any fixed $k\geq1$, $\sup_{0\leq t\leq T}\vert X_{t}^{(n)}(x)\vert^{k}$
are uniformly integrable. This implies that for any fixed $r>0$,
$\{X_{t}^{(n)}\}_{n}$ is a Cauchy sequence in $L^{k}(\Omega\times B_{r};C([0,T]))$
and we denote the limit as $X_{t}$. Since each $X_{t}^{(n)}$ satisfies
Definition \ref{def: almost everywhere stochastic flow}, it is easy
to verify that their limit $X_{t}$ also satisfies (1) and (2) of
Definition \ref{def: almost everywhere stochastic flow}.

Step 2: Now we verify that the limit $X_{t}$ satisfies (3) of Definition
\ref{def: almost everywhere stochastic flow}, i.e. if we define 
\[
Y_{t}(x)=x+\int_{0}^{t}b(s,X_{s}(x))ds+\int_{0}^{t}dB_{s},
\]
then $X_{t}=Y_{t}$ in space $L^{1}(\Omega\times B_{r};C([0,T]))$.
Consider 
\[
\sup_{0\leq t\leq T}\vert X_{t}^{n}(x)-Y_{t}(x)\vert\leq\int_{0}^{T}\vert b^{n}(t,X_{t}^{(n)}(x))-b(t,X_{t}(x))\vert dt.
\]
Then integrate both sides on $B_{r}$ and take expectation, we have
that 
\begin{align*}
\mathbb{E}\left[\int_{B_{r}}\sup_{0\leq t\leq T}\vert X_{t}^{n}(x)-Y_{t}(x)\vert dx\right] & \leq\mathbb{E}\left[\int_{B_{r}}\int_{0}^{T}\vert b^{n}(t,X_{t}^{(n)}(x))-b(t,X_{t}^{(n)}(x))\vert dtdx\right]\\
 & \quad+\mathbb{E}\left[\int_{B_{r}}\int_{0}^{T}\vert b(t,X_{t}^{(n)}(x))-b(t,X_{t}(x))\vert dtdx\right]\\
 & =I_{1}+I_{2}.
\end{align*}
Again by (2) in Definition \ref{def: almost everywhere stochastic flow},
we have that $I_{1}\leq(1+\vert B_{r}\vert)\Vert b-b^{(n)}\Vert_{L_{t}^{1}L_{x}^{p}}$.
For the second term, we will find another smooth $b_{\epsilon}$ such
that $\Vert b_{\epsilon}-b\Vert_{L_{t,x}^{p}}\leq\epsilon$ and then
separate $I_{2}$ into three parts
\begin{align*}
I_{2} & \leq\mathbb{E}\left[\int_{B_{r}}\int_{0}^{T}\vert b_{\epsilon}(t,X_{t}^{(n)}(x))-b_{\epsilon}(t,X_{t}(x))\vert dtdx\right]\\
 & \quad+\mathbb{E}\left[\int_{B_{r}}\int_{0}^{T}\vert b_{\epsilon}(t,X_{t}^{(n)}(x))-b(t,X_{t}^{(n)}(x))\vert dtdx\right]\\
 & \quad+\mathbb{E}\left[\int_{B_{r}}\int_{0}^{T}\vert b_{\epsilon}(t,X_{t}(x))-b(t,X_{t}(x))\vert dtdx\right].
\end{align*}
For the second and the third terms, we control them just as $I_{1}$.
The first term converges to $0$ as $n\rightarrow0$ since $X_{t}^{(n)}\rightarrow X_{t}$
in $L^{k}(\Omega\times B_{r};C([0,T]))$ and now we can conclude that
$X_{t}=Y_{t}$ $\mathbb{P}\times m$-almost everywhere.

Step 3: Finally we prove that the solution is unique. Suppose that
we have two almost everywhere stochastic flows $X_{t}$ and $\tilde{X}_{t}$
corresponding to the same $b\in L^{1}(0,T;W^{1,p}(\rd))\cap L^{l}(0,T;L^{q}(\rd))$
and we apply Lemma \ref{lem: the log estiamte} to them to deduce
that 
\[
\mathbb{E}\left[\int_{B_{r}}\log\left(\frac{\sup_{0\leq t\leq T}\vert X_{t}(x)-\tilde{X}_{t}(x)\vert^{2}}{\theta^{2}}+1\right)dx\right]\leq C\Vert\nabla b\Vert_{L_{t}^{1}L_{x}^{p}},
\]
which is uniform for all $\theta>0$. Hence we can take $\theta\rightarrow0$
and now the proof is complete. 
\end{proof}
\textbf{Acknowledgment:} The Authors would like to thank our supervisor
Professor Zhongmin Qian for bringing these questions to our attention
and discussing constantly with us. We also want to thank the reviewers
for pointing out an alternative proof of Theorem \ref{thm: SDE main theorem}
using Kolmogorov's continuity theorem.

\bibliographystyle{plain}
\bibliography{MaximumInequality}

\end{document}